\documentclass[12pt]{amsart}
\usepackage{amssymb}
\usepackage{amsmath, amscd}
\usepackage{amsthm}
\usepackage{amstext}
\usepackage[table,xcdraw]{xcolor}
\usepackage{ulem}
\usepackage{tikz}
\usepackage{circuitikz}
\usepackage[colorlinks=true, linkcolor=blue,urlcolor=blue]{hyperref}

\newtheorem{theorem}{Theorem}[section]

\newtheorem{proposition}[theorem]{Proposition}
\newtheorem{lemma}[theorem]{Lemma}

\newtheorem{corollary}[theorem]{Corollary}
\theoremstyle{definition}

\newtheorem{example}[theorem]{Example}

\begin{document}

\author[A. Javan]{Arash Javan}
\address{Department of Mathematics, Tarbiat Modares University, 14115-111 Tehran Jalal AleAhmad Nasr, Iran}
\email{a.darajavan@modares.ac.ir; a.darajavan@gmail.com}

\author[A. Moussavi]{Ahmad Moussavi}
\address{Department of Mathematics, Tarbiat Modares University, 14115-111 Tehran Jalal AleAhmad Nasr, Iran}
\email{moussavi.a@modares.ac.ir; moussavi.a@gmail.com}

\author[Peter Danchev]{Peter Danchev}
\address{Institute of Mathematics and Informatics, Bulgarian Academy of Sciences, 1113 Sofia, Bulgaria}
\email{danchev@math.bas.bg; pvdanchev@yahoo.com}

\title{Generalizing Semi-$n$-Potent Rings}

\maketitle




\begin{abstract}
We define and explore the class of rings $R$ for which each element in $R$ is a sum of a tripotent element from $R$ and an element from the subring $\Delta(R)$ of $R$ which commute each other. Succeeding to obtain a complete description of these rings modulo their Jacobson radical as the direct product of a Boolean ring and a Yaqub ring, our results somewhat generalize those established by Ko\c{s}an-Yildirim-Zhou in Can. Math. Bull. (2019).
\end{abstract}

\section{Introduction and Motivation}

Throughout this paper, all rings are assumed to be unital and associative. Almost all symbols, notation and concepts are standard being consistent with the classical book \cite{lamf}. The Jacobson radical, the lower nil-radical, the set of nilpotent elements, the set of idempotent elements, and the set of units of \( R \) are denoted, respectively, by \( J(R) \), \( \text{Nil}_{*}(R) \), \( \text{Nil}(R) \), \( \text{Id}(R) \), and \( U(R) \). Additionally, we write ${\rm M}_n(R)$, ${\rm T}_n(R)$ and $R[x]$ for the $n \times n$ full matrix ring, the $n \times n$ upper triangular matrix ring, and the polynomial ring over $R$, respectively.

The core focus of this exploration is the set
\begin{align*}
J(R) \subseteq \Delta(R) &= \{ x \in R : x + u \in U(R) \text{ for all } u \in U(R) \} \\
                           &= \{ x \in R : 1 - xu \text{ is invertible for all } u \in U(R) \} \\
                           &= \{ x \in R : 1 - ux \text{ is invertible for all } u \in U(R) \},
\end{align*}
which was examined by Lam \cite[Exercise 4.24]{lame} and recently explored in detail by Leroy-Matczuk \cite{lmr}. It was indicated in \cite[Theorems 3 and 6]{lmr} that \( \Delta(R) \) represents the (proper) largest Jacobson radical subring of \( R \) that remains closed under multiplication by all units (resp., quasi-invertible elements) of \( R \), and it is an ideal of $R$ exactly when $\Delta(R)=J(R)$.

In the contemporary ring theory, the class of strongly nil-clean rings possess significant importance. A ring \(R\) is called \textit{strongly nil-clean} if every element of \(R\) can be expressed as the sum of an idempotent in $R$ and a nilpotent element in $R$ that commute with each other (see \cite{chenb, diesl, kzw}). Later on, Chen and Sheibani in \cite{css} generalized this concept and introduced the so-called strongly 2-nil-clean rings: a ring \(R\) is \textit{strongly 2-nil-clean} if every element of \(R\) can be written as the sum of a tripotent element of $R$ (i.e., an element \(x\in R\) such that \(x^3 = x\)) and a nilpotent element of $R$ that commute.

On the other hand, in a way of similarity, \textit{strongly J-clean} rings are rings in which every element can be written as the sum of an idempotent and an element from the Jacobson radical that commute (see \cite{csj, chensjc}). In this vein, Ko\c{s}an et al. introduced in \cite{kyzr} the so-termed \textit{semi-tripotent} rings $R$ in which each element is the sum of a tripotent element from $R$ and an element from \(J(R)\).

Considering and analyzing these definitions, as well as the fact that \(\Delta(R)\) is a (possibly proper) subset of \(J(R)\), that is {\it not} necessarily an ideal, and which also does {\it not} have useful properties like the set \(\text{Nil}(R)\), a question naturally arises about the properties of those rings $R$ for which each element is the sum of a tripotent element from $R$ and an element from \(\Delta(R)\) that commute with each other. The main objective of the current article is namely to investigate these types of rings and to conduct a comprehensive study of their structure.

Thereby, we come to the following key notion, motivated from the discussion alluded to above.

\medskip

\noindent{\bf Definition.} We say that \( R \) is a {\it strongly $\Delta$ tripotent} ring, or just an {\it SDT ring} for short, if every element of \( R \) is the sum of a tripotent from $R$ and an element from \( \Delta(R) \) that commute with each other. Such a sum's presentation is also said to be an {\it SDT representation}.

\medskip

Our further plan in organization of our study is the following: In the next section, we obtain some crucial examples and principal properties of such rings establishing their connection with many standard properties -- e.g., such as uniquely clean (see Corollary~\ref{cor 2}). In the subsequent section, we achieve the major result describing the algebraic structure of the SDT rings in an appropriate form showing that these rings modulo their Jacobson radical are the direct product of a Boolean ring and a Yaqub ring (see Theorem~\ref{boolean and yaqub}). Some other closely related statements are also proved such as Propositions~\ref{local} and \ref{semitrip}. In the fourth section, we study the behavior of the given SDT concept under various ring extensions and, specifically, we characterize when ${\rm T}_n(R)$ is a SDT ring by finding a necessary and sufficient condition, provided that $R$ is local and $n\geq 3$ (see Theorem~\ref{local}). In the final fifth section, we conclude with some commentaries and three open problems (see, e.g., Problems 1, 2 and 3) which, hopefully, will stimulate a future intensive examination of the present subject.

\section{Examples and Basic Properties}

The following claim can easily be proven, so we omit the details leaving them to the interested reader for check.

\begin{lemma}\label{lemma 0}
(1) Suppose \( R = \prod_{i \in I} R_i \). Then, \( R \) is an SDT ring if, and only if, for each \( i \in I \), \( R_i \) is an SDT ring.

(2) Suppose \( R \) is a ring and \( I \) is an ideal of \( R \) such that \( I \subseteq J(R) \). Then, \( R/I \) is an SDT ring.
\end{lemma}

We proceed by proving with the following three technical assertions.

\begin{lemma}\label{lemma 1}
 For every \( e = e^3 \in R \) and \( d \in \Delta(R) \), we have \( (e \pm e^2)d, d(e \pm e^2), 2e^2d, \) and \( 2ed \in \Delta(R) \).
\end{lemma}

\begin{proof}
For every \( e = e^3 \in R \), we have \[((1-e^2) - e)((1-e^2) - e) = 1 = ((1-e^2) + e)((1-e^2) + e).\] Therefore, \((1-e^2 \pm e) \in U(R)\), so it follows from \cite[Lemma 1(2)]{lmr} that, for every \( d \in \Delta(R) \), both \(((1-e^2 \pm e)d \) and \( d((1-e^2 \pm e) \in \Delta(R)\). Since \( \Delta(R) \) is a subring of \( R \), we have \((e \pm e^2)d \) and \( d(e \pm e^2) \in \Delta(R) \). This implies that \( 2ed \) and \( 2e^2d \in \Delta(R) \), as required.
\end{proof}

\begin{lemma} \label{lemma 2}
Let \( R \) be an SDT ring, and \( a \in R \). If \( a^2 \in \Delta(R) \), then \( a \in \Delta(R) \).
\end{lemma}

\begin{proof}
Assume that \( a = e + d \) is an SDT representation. We have \( a^2 = e^2 + 2ed + d^2 \). By Lemma \ref{lemma 1}, it must be that \[ e^2 = a^2 - 2ed - d^2 \in \Delta(R) \cap \text{Id}(R) = \{0\} ,\] which implies \( e = 0 \). Thus, \( a = d \in \Delta(R) \), as expected.
\end{proof}

\begin{lemma} \label{lemma 3}
Let \( R \) be an SDT ring. Then, for every \( a \in R \), \( a - a^3 \in \Delta(R) \).
\end{lemma}

\begin{proof}
Assume \( a = e + d \) is an SDT representation. We calculate that

\[ a - a^3 = (d - d^3) - (2e^2d + 2ed) - (e^2d + ed^2). \]

Furthermore, according to Lemma \ref{lemma 1}, it suffices to show that \( e^2d + ed^2 \in \Delta(R) \). But, Lemma \ref{lemma 1} tells us that \( e^2d + ed^2 \in \Delta(R) \) precisely when \( ed + e^2d^2 \in \Delta(R) \). Consequently, we show that \( ed + e^2d^2 \in \Delta(R) \).

To this target, assume \( ed = f + b \) is an SDT representation. Then,

\[ e^2d^2 = f^2 + 2fb + b^2. \]

Thus,

\[ ed + e^2d^2 = (f + f^2) + (b + 2fb + b^2). \]

Now, multiplying by \( d \) and \( d^2 \) both sides in the previous relation, we have:

\[ ed^2 + e^2d^3 = (f + f^2)d + (b + 2fb + b^2)d \in \Delta(R), \]

\[ ed^3 + e^2d^4 = (f + f^2)d^2 + (b + 2fb + b^2)d^2. \]

Owing to Lemma \ref{lemma 1}, we infer that \( ed^2 + e^2d^3, ed^3 + e^2d^4 \in \Delta(R) \). Also,

\[ ed^2 + e^2d^3 = ed^2 + e^2d^3 - ed^3 + ed^3 - e^2d^2 + e^2d^2 = e^2d^2 + ed^3 + (e^2 - e)d^3 + (e - e^2)d^2. \]

Thus, in virtue of Lemma \ref{lemma 1}, it follows that \( e^2d^2 + ed^3 \in \Delta(R) \). Therefore, we get

\[
\begin{cases}
e^2d^2 + ed^3 \in \Delta(R) \\
ed^3 + e^2d^4 \in \Delta(R)
\end{cases}
\Longrightarrow e^2d^2 + e^2d^4 \in \Delta(R).
\]

We now have that \[(ed + e^2d^2)^2 = e^2d^2 + 2ed^3 + e^2d^4 \in \Delta(R).\] So, Lemma \ref{lemma 2} enables us that \( ed + e^2d^2 \in \Delta(R) \), as pursued.
\end{proof}

We now arrive at the following concrete application of the last lemma.

\begin{example} \label{example 1}
Let \( R \) be an arbitrary ring. Then, \( R[x] \) is {\it not} an SDT ring.
\end{example}

\begin{proof}
Assume the contrary. Then, applying Lemma \ref{lemma 3}, we derive that \( x - x^3 \in \Delta(R[x]) \), and thus \( 1 - x + x^3 \in U(R[x]) \), which is the wanted contradiction.
\end{proof}

With the previous example in mind, the ring \( R[x] \) is surely not SDT. However, a logical question arises about the form of elements with an SDT representation in the polynomial ring \( R[x] \). We will attempt to answer this question below.

\medskip

Recall that a ring $R$ is said to be {\it 2-primal} if ${\rm Nil}_{*}(R) = {\rm Nil}(R)$. For instance, it is well known that any commutative ring and any reduced ring are definitely 2-primal.

Likewise, for an endomorphism \( \sigma \) of \( R \), the ring \( R \) is called {\it \(\sigma\)-compatible} if, for every \( a, b \in R \), the equality \( ab = 0 \) if, and only if, \( a\sigma(b) = 0 \) (see \cite{ann}). In this case, it is clear that \( \sigma \) is always injective.

\medskip

We now manage to prove the following two pivotal statements.

\begin{proposition} \label{pro 1}
Let \( R \) be a 2-primal and \(\alpha\)-compatible ring. Then, \[\Delta(R[x, \alpha]) = \Delta(R) + Nil_{*}(R[x, \alpha])x.\]
\end{proposition}

\begin{proof}
Assuming \( f = \sum_{i=0}^{n} a_i x^i \in \Delta(R[x, \alpha]) \), then, for each \( u \in U(R) \), we have that \( 1 - uf \in U(R[x, \alpha]) \). Thus, taking into account \cite[Corollary 2.14]{cheni}, \( 1 - ua_0 \in U(R) \) holds and, for every \( 1 \le  i \le  n \), it holds \( ua_i \in {\rm Nil}_{*}(R) \). Since \( {\rm Nil}_{*}(R) \) is an ideal, we deduce \( a_0 \in \Delta(R) \) and hence, for each \( 1 \le  i \le  n \), we obtain \( a_i \in Nil_{*}(R) \). Since \( R \) is a 2-primal ring, \cite[Lemma 2.2]{cheni} applies to get that \( Nil_{*}(R)[x,\alpha] = Nil_{*}(R[x, \alpha]) \), as desired.

Conversely, assume \( f \in \Delta(R) + Nil_{*}(R[x, \alpha])x \) and \( u \in U(R[x,\alpha]) \). Then, employing \cite[Corollary 2.14]{cheni}, we have \( u \in U(R) + Nil_{*}(R[x, \alpha])x \). But, since \( R \) is a 2-primal ring, we receive \( 1 - uf \in U(R) + Nil_{*}(R[x, \alpha])x \subseteq U(R[x, \alpha]) \), whence \( f \in \Delta(R[x, \alpha]) \), as promised.
\end{proof}

\begin{proposition}\label{pro 2}
Let \( R \) be a 2-primal and \(\alpha\)-compatible ring, and let \( e^3 = e = \sum_{i=0}^n e_i x^i \in R[x, \alpha] \). Then, \( e_0^3 = e_0 \) and, for every \( 1 \le i \le n \), the inclusion \( e_i \in Nil(R) \) is true.
\end{proposition}

\begin{proof}
It is easy to see that \( e_0^3 = e_0 \), so it suffices to show that, for every \( 1 \le i \le n \), the relation \( e_i \in {\rm Nil}(R) \) is valid. Since \( e^3 = e \), we inspect that \( e_n \alpha^n(e_n) \alpha^{2n}(e_n) = 0 \). And because \( R \) is \(\alpha\)-compatible, \cite[Lemma 2.1]{hm} is applicable to get that \( e_n^3 = 0 \).

Now, set \( g := f - e_n x^n \). Since \( f^3 = f \) and \( e_n \in {\rm Nil}_*(R) \), we have \( g - g^3 \in {\rm Nil}_*(R)[x, \alpha] \), so \(\bar{g} = \bar{g}^3 \in R/{\rm Nil}_*(R)[x, \alpha] \). Thus, one verifies that
\[e_{n-1} \alpha^{n-1}(e_{n-1}) \alpha^{2n-2}(e_{n-1}) \in Nil_*(R).\] But, since \( R \) is an \(\alpha\)-compatible ring, \cite[Lemma 2.1]{hm} works to obtain that \( e_{n-1} \in {\rm Nil}(R) \). Continuing in this aspect, it can be shown that, for each \( 1 \le i \le n \), the condition \( e_i \in Nil(R) \) is fulfilled, as asked for.
\end{proof}

To specify the elements with an SDT representation of the ring \( R[x, \alpha] \), we need new notation. For convenience of the exposition, we just put the set of elements with an SDT representation in the ring \( R \) to be abbreviated as \( SDT(R) \).

\medskip

So, we have the validity of the following.

\begin{lemma}
Let \( R \) be a 2-primal and \(\alpha\)-compatible ring. Then, \[ SDT(R[x, \alpha]) \subseteq SDT(R) + Nil_*(R)[x, \alpha]x .\]
\end{lemma}

\begin{proof}
Assume \( f = \sum_{i=0}^n f_i x^i \in SDT(R[x, \alpha]) \) and \( f = \sum_{i=0}^n e_i x^i + \sum_{i=0}^n d_i x^i \) is an SDT representation. In accordance with Propositions \ref{pro 1} and \ref{pro 2}, we have \( e_0 = e_0^3 \) and \( d_0 \in \Delta(R) \), and hence clearly \( e_0 d_0 = d_0 e_0 \), so that \( f_0 \in SDT(R) \).

Moreover, with the aid of Proposition \ref{pro 2}, for every \( 1 \le i \le n \), it must be that \( e_i, d_i \in {\rm Nil}_*(R) \), whence \( f_i = e_i + d_i \in {\rm Nil}_*(R) \), as required.
\end{proof}

The next affirmation is crucial.

\begin{lemma} \label{lemma 4}
Let \( R \) be an SDT ring. Then, \( R/J(R) \) is reduced.
\end{lemma}

\begin{proof}
Assume \( x^2 \in J(R) \subseteq \Delta(R) \). Thus, by Lemma \ref{lemma 2}, we have \( x \in \Delta(R) \). Let \( r \in R \). Since \( 1 - r^2x^2 \in U(R) \), we may set \( u := 1 - rx^2r \in U(R) \). Therefore,

\[
(1 - rx)(1 + rx) = 1 - rx + xr - rx^2r = xr - rx + u.
\]

It suffices to show that \( xr - rx \in \Delta(R) \). To this goal, assume \( r = e + d \) is an SDT representation. Then,

\[
xr - rx = x(e + d) - (e + d)x = xe - ex + (xd - dx),
\]

and as \( x, d \in \Delta(R) \), it is just sufficient to prove that \( xe - ex \in \Delta(R) \).

Since \[\left[ e^2x(1 - e^2) \right]^2 = 0 = \left[ (1 - e^2)xe^2 \right]^2,\] Lemma \ref{lemma 2} assures that

\[
\begin{cases}
e^2x(1 - e^2) \in \Delta(R) \Longrightarrow e^2x - e^2xe^2 \in \Delta(R).\\
(1 - e^2)xe^2 \in \Delta(R) \Longrightarrow xe^2 - e^2xe^2 \in \Delta(R).
\end{cases}
\]

However, because \(\Delta(R)\) is closed under addition, we arrive at \( e^2x - xe^2 \in \Delta(R) \). Consequently,

\[
xe - ex =xe+xe^2-xe^2-ex-e^2x+e^2x= e^2x-xe^2+x(e + e^2)-(e + e^2)x \in \Delta(R).
\]

Hence,

\[
(1 - rx)(1 + xr) \in U(R).
\]

But \( R \) was arbitrary, and so \( x \in J(R) \), as needed.
\end{proof}

Given the truthfulness of Lemma \ref{lemma 3}, we have that, for every SDT ring \( R \), \( 6 = 2^3 - 2 \in \Delta(R) \). This raises a logical question: if \( R \) is an SDT ring, is \( 6 \in J(R) \)? We will answer this query in the following lemma.

\begin{lemma}\label{lemma 5}
Let \( R \) be an SDT ring. Then, \( 6 \in J(R) \).
\end{lemma}

\begin{proof}
Invoking Lemma \ref{lemma 3}, we know that \( 6 \in \Delta(R) \), which implies \( 12 = 6 + 6 \in \Delta(R) \). Letting \( r \in R \) be arbitrary, and letting \( r = e + d \) be an SDT representation, Lemma \ref{lemma 1} ensures that

\[
1 - 12r = 1 - 12e - 12d = 1 - 2(6e) - 12d \in 1 + \Delta(R) \subseteq U(R).
\]

Thus, \( 12 \in J(R) \).

Furthermore, since \( 6^2 = 36 = 3 \times 12 \in J(R) \), Lemma \ref{lemma 4} helps us to conclude that \( 6 \in J(R) \), as stated.
\end{proof}

As a useful consequence, we deduce the following.

\begin{corollary}\label{cor 1}
Let $R$ be an SDT ring. Then, the following two points hold:

(1) $2 \in U(R)$ if, and only if, $3 \in J(R)$.

(2) $3 \in U(R)$ if, and only if, $2 \in J(R)$.
\end{corollary}

\begin{proof}
The proof is pretty straightforward being based on Lemma \ref{lemma 5}, so we leave it voluntarily.
\end{proof}

The next two assertions are worthy of documentation.

\begin{proposition}\label{pro 3}
Let $R$ be an SDT ring such that $2 \in U(R)$. Then, $\Delta(R)$ is an ideal. In particular, under these conditions, $\Delta(R) = J(R)$.
\end{proposition}

\begin{proof}
Since $\Delta(R)$ is closed under addition, it is sufficient to show that, for any $d \in \Delta(R)$ and $r \in R$, the relations $rd, dr \in \Delta(R)$ are valid. Assume, for this aim, that $rd = e + b$ and $r = f + b'$ are two SDT representations. Exploiting Lemma \ref{lemma 1}, we know $2fd \in \Delta(R)$. Since $2 \in U(R)$, \cite[Lemma 1(2)]{lmr} insures that $fd \in \Delta(R)$. So, we have

$$rd = e + b = fd + b'd  \Longrightarrow  e - fd = b'd - b \in \Delta(R).$$

\noindent But, since $fd \in \Delta(R)$, it follows that $e \in \Delta(R)$, so $e^2 \in \Delta(R) \cap Id(R) = \{0\}$, which forces $e = 0$. Therefore, $rd = b \in \Delta(R)$. Similarly, it can be shown that $dr \in \Delta(R)$, guaranteeing the claim.
\end{proof}

\begin{proposition}\label{pro 4}
Let \(R\) be an SDT ring with \(3 \in U(R)\). Then, for any \(a \in R\), we have \(a = f + b\), where \(f = f^2 \in R\), \(b \in \Delta(R)\) and \(fb = bf\).
\end{proposition}

\begin{proof}
Suppose \(a = f + d\) is an SDT representation. Then, \[a - a^2 = (f - f^2) + (d - 2fd - d^2).\] Since \(3 \in U(R)\) by Corollary \ref{cor 1}, we get \(2 \in J(R)\). Thus, \((f - f^2)^2 = -2(f - f^2) \in J(R)\) and, with Lemma \ref{lemma 5} at hand, we observe that \(f - f^2 \in J(R)\). This gives \(a - a^2 \in \Delta(R)\).

On the other hand, since \[a - f^2 = (a - a^2) + 2(a^2 - f^2 - fd) - d^2 \in \Delta(R),\] by setting \(e := f^2\) we finish the proof after all.
\end{proof}

A ring \(R\) is called an {\it SDI ring} if, for every \(r \in R\), there exist \(e = e^2 \in R\) and \(b \in \Delta(R)\) such that \(r = e + b\) and \(eb = be\). Recall also that a ring is called a {\it \(\Delta U\) ring}, provided \(1 + \Delta(R) = U(R)\) (see \cite{kkqt}).

\medskip

The following closely related results are of some interest as well.

\begin{lemma}
Every SDI ring is a \(\Delta U\) ring.
\end{lemma}

\begin{proof}
Suppose \(u \in U(R)\) and \(u = e + d\) is an SDI representation. Then, we have \[e = u - d \in U(R) + \Delta(R) \subseteq U(R) \cap Id(R) = \{1\},\] as required.
\end{proof}

\begin{lemma}\cite[Proposition 2.3]{kkqt} \label{sum unit}
The ring $R$ is a $\Delta U$ ring if, and only if, $U(R) + U(R) \subseteq \Delta(R)$; and then, $U(R) + U(R) = \Delta(R)$.
\end{lemma}

Recall that a ring $R$ is said to be {\it uniquely clean}, provided that each element in $R$ has a unique representation as the sum of an idempotent and a unit (see \cite{nzu}).

\medskip

The next valuable consequence gives some transversal between the notions of SDI rings and unique cleanness.

\begin{corollary}\label{cor 2}
Let $R$ be a ring. Then, the following are equivalent:

(1) $R$ is uniquely clean.

(2) $R$ is SDI with all idempotents central.
\end{corollary}

\begin{proof}
$(1) \Rightarrow (2)$. Assume \(R\) is a uniquely clean ring. Consulting with \cite[Lemma 4]{nzu}, every idempotent in \(R\) is central. Besides, by virtue of \cite[Theorem 20]{nzu}, for every \(a \in R\), there exists a unique idempotent \(e\) such that \(a - e \in J(R) \subseteq \Delta(R)\). Thus, there exists \(d \in \Delta(R)\) such that \(a = e + d\). Since all idempotents are central, we have \(ed = de\).

$(2) \Rightarrow (1)$. Assume \(R\) is an SDI ring, and let \(a \in R\) be arbitrary. Suppose \(a+1 = e + d\) is an SDI representation. Then, \(a = e + (d-1)\), which is a clean representation. Assume now that \(e + u = f + v\) are two clean representations. So, Lemma \ref{sum unit} informs us that \(e - f = v - u \in \Delta(R)\). Since all idempotents are central, we find \(e - f = (e-f)^3\), and so \((e-f)^2 \in \Delta(R) \cap Id(R) = \{0\}\). Therefore, \(e - f = (e-f)^3 = (e-f)(e-f)^2 = 0\). Hence, \(e = f\), as it must be.
\end{proof}

\section{The Main Characterization}

We start our considerations here with some relationships between certain classes of rings.

\begin{proposition}\label{local} Suppose $R$ is an SDT ring and a domain. Then, $R$ is a local ring.
\end{proposition}

\begin{proof} Let $a \in R$. We want to show that either $a \in U(R)$ or $a \in \Delta(R)$. To that end, suppose $a = e + d$ is an SDT representation. If $e = 0$, then $a = d \in \Delta(R)$. If $e \neq 0$, then as $e^3 = e$ it must be $e(1-e^2) = 0$. But, since $R$ is a domain, $(1-e)(1+e) = 1-e^2 = 0$, so either $e = 1$ or $e = -1$. Therefore, either $a = 1 + d \in U(R)$ or $a = -1 + d \in U(R)$. It next can easily be shown that $R$ is a local ring if, and only if, $R = U(R) \cup \Delta(R)$, as required.
\end{proof}

As an immediate consequence, we yield:

\begin{corollary} Suppose $R$ is a strongly 2-nil clean and local ring. Then, $R$ is an SDT ring.
\end{corollary}

\begin{proof} It is pretty easy, because in a local ring the containment $\text{Nil}(R) \subseteq J(R)$ always holds.
\end{proof}

The next assertion is of some importance by giving some close relevance between the notion of a semi-tripotent ring as stated in \cite{kyzr} and the new concept of an SDT ring given above.

\begin{proposition}\label{semitrip} Suppose $R$ is a semi-tripotent and local ring. Then, $R$ is an SDT ring.
\end{proposition}

\begin{proof} Since $R$ is a local ring, either $2 \in J(R)$ or $2 \in U(R)$. If $2 \in J(R)$, then in virtue of \cite[Theorem 3.5]{kyzr} the factor-ring $R/J(R)$ is Boolean. On the other hand, as $R$ is local, it has to be that $R/J(R) \cong \mathbb{Z}_2$, and so $R = J(R) \cup (1 + J(R))$, yielding $R$ is an SDT ring. If, however, $2 \in U(R)$, then again \cite[Theorem 3.5]{kyzr} works to get that the quotient-ring $R/J(R)$ is a Yaqub ring. However, because $R$ is local, it must be that $R/J(R) \cong \mathbb{Z}_3$, and thus $R = J(R) \cup (1 + J(R)) \cup (-1 + J(R))$ implying $R$ is an SDT ring, as asserted.
\end{proof}

It is well known that a ring is Boolean if, and only if, it is a subdirect product of copies of \( \mathbb{Z}_2 \). Analogously, in \cite{css}, Chen and Sheibani called a non-zero ring \(R\) a {\it Yaqub ring} if it is a subdirect product of copies of \( \mathbb{Z}_3 \). They proved that \(R\) is a Yaqub ring if, and only if, 3 is nilpotent and \(R\) is a tripotent ring (that is, each its element is tripotent).

\medskip

We are now ready to attack the chief characterizing result, thereby completely describing the structure of the SDT rings.

\begin{theorem}\label{boolean and yaqub}
Assume \(R\) is an SDT ring. Then, \(R/J(R)\) is a tripotent ring, i.e., \(R/J(R) \cong R_1 \times R_2\), where \(R_1\) is a Boolean ring and \(R_2\) is a Yaqub ring.
\end{theorem}

\begin{proof}
Referring to Lemma \ref{lemma 5}, we have \(6 \in J(R)\). Set \(\bar{R} := R/J(R)\). Thanks to the famous Chinese Remainder Theorem, we write \(\bar{R} \cong R_1 \times R_2\), where \(R_1 := \bar{R}/2\bar{R}\) and \(R_2 := \bar{R}/3\bar{R}\). Since \(R\) is an SDT ring, Lemma \ref{lemma 0}(2) guarantees that \(\bar{R}\) is an SDT ring too. Therefore, again in view of Lemma \ref{lemma 0}(1), \(R_1\) is an SDT ring. Since \(2 = 0\) in \(R_1\), we have \(3 \in U(R_1)\). Thankfully, Proposition \ref{pro 4} yields \(R_1\) is an SDI ring. Also, Lemma \ref{lemma 4} implies that \(R_1\) is reduced, and thus all idempotents in \(R_1\) are central. Therefore, Corollary \ref{cor 2} shows that \(R_1\) is a uniquely clean ring. Note that, as \(J(R) = 0\), it must be that \(J(R_1) = 0\). Using now \cite[Theorem 19]{nzu}, we conclude that \(R_1\) is a Boolean ring, as formulated.

On the other hand, since \(3 = 0\) in \(R_2\not=\{0\}\), we have \(2 \in U(R_2)\). Knowing Proposition \ref{pro 3}, we obtain \(J(R_2) = \Delta(R_2)\). This means, with the help of Lemma \ref{lemma 3}, that, for any \(a \in R_2\), the relations \(a - a^3 \in \Delta(R_2) = J(R_2) = 0\) are true. Thus, for any \(a \in R_2\), we get that \(a = a^3\). Furthermore, using \cite[Lemma 4.4]{css}, we infer that \(R_2\) is a Yaqub ring, as given.
\end{proof}

It is worthwhile noticing that the extra requirement on the first direct component $R_1$ and the second direct component $R_2$ to be not simultaneously $\{0\}$ can be freely ignored here, as opposite to the showed in \cite{p}, where an analogous shortcoming was unambiguously detected for the main result of the paper \cite{ykz}.

\medskip

Let $R$ be a ring, and let $a \in R$. Suppose $ann_{l}a := \{ r \in R : ra = 0\}$ and $ann_{r}a := \{r \in  R: ar = 0\}$.

\medskip

We continue by verifying the following two needed technicalities.

\begin{lemma}\label{ann}
Let $R$ be a ring and $a = e + d$ an SDT representation in $R$. Then, $\text{ann}_l(a) \subseteq \text{ann}_l(e)$ and $\text{ann}_r(a) \subseteq \text{ann}_r(e)$.
\end{lemma}

\begin{proof}
Assume $ra = 0$. Now, Lemma \ref{lemma 1} applies to ensure that there exists $d' \in \Delta(R)$ such that $a^2 = e^2 + d'$. Since $ra = 0$, we have $re^2 + rd' = 0$. Now, multiplying by $e$ from the right, we get $re + red' = 0$, and so $re(1 + d') = 0$. Since $d' \in \Delta(R)$, it follows that $1 + d' \in U(R)$ which forces $re = 0$. Thus, $r \in \text{ann}_l(e)$. Similarly, it can be shown that the inclusion $\text{ann}_r(a) \subseteq \text{ann}_r(e)$ is too valid, as required.
\end{proof}

\begin{lemma}
Let $R$ be a ring and $e \in R$ an idempotent. If $a \in eRe$ is an SDT element in $R$, then $a$ is an SDT element in the corresponding corner subring $eRe$.
\end{lemma}

\begin{proof}
Write $a = f + d$, where $f = f^3$, $d \in \Delta(R)$ and $fd = df$. Since $1 - e \in \text{ann}_l(a) \cap \text{ann}_r(a)$, Lemma \ref{ann} is a guarantor that $1 - e \in \text{ann}_l(f) \cap \text{ann}_r(f)$ implying $(1 - e)f = f(1 - e) = 0$. Thus, $f = ef = fe$. Likewise, since $a \in eRe$, we receive $a = ea = ae = eae$. But, subsequently multiplying $a = f + d$ by $e$ from the left and right, we obtain that $a = efe + ede$. Note that, since $f = ef = fe$ and $f$ is a tripotent, $efe$ is also a tripotent. So, it suffices to show that $ede \in \Delta(eRe)$.

On the other hand, since $f = ef = fe = efe$ and $a = ea = ae = eae$, it is evident that $$d = ed = de = ede \in \Delta(R) \cap eRe.$$ Now, we show that $eRe \cap \Delta(R) \subseteq \Delta(eRe)$ always holds. To this purpose, assume $r \in eRe \cap \Delta(R)$ and $u \in U(eRe)$. Then, $(u + (1 - e))(u^{-1} + (1 - e)) = 1$, so $u + (1 - e) \in U(R)$. Since $r \in \Delta(R)$, there exists $v \in R$ such that $(1 - (u + (1 - e))r)v = 1$. But $r \in eRe$, so that $(1 - ur)v = 1$. Furthermore, multiplying subsequently by $e$ from the left and right, we extract that $(e - ur)eve = e$ forcing $r \in \Delta(eRe)$. Finally, $d \in \Delta(eRe)$, and we are done.
\end{proof}

As an automatic consequence, we yield the following.

\begin{corollary}\label{corner ring}
Let $R$ be a ring, and let $e \in R$ be an idempotent. If $R$ is SDT ring, then so is the corresponding corner subring $eRe$.
\end{corollary}

\section{Triangular Matrix Rings}

As usual, a ring $R$ is termed {\it local}, provided $R/J(R)$ is a division ring, that is, each element in $R\setminus J(R)$ is a unit, which set-theoretically means that $R=J(R)\cup U(R)$.

\medskip

We begin here with the following technicality.

\begin{lemma}\label{trivial tripotent}
Let $R$ be a local ring with $2 \in U(R)$. Then, $R$ has only trivial tripotent elements.
\end{lemma}

\begin{proof}
Suppose that $e = e^3 \in R$. If $e \in J(R)$, then $e(1 - e^2) = 0$, whence $e = 0$. If now $e \in U(R)$, then $e^2 = 1$, and so $(1 - e)(1 + e) = 0$. Since $(1 - e) + (1 + e) = 2 \in U(R)$ and $R$ is a local ring, we have either $1 - e \in U(R)$ or $1 + e \in U(R)$. This, in turn, means that either $e = 1$ or $e = -1$, as required.
\end{proof}

Based on the above claim, we now extend the well-known Workhorse Lemma (see \cite[Lemma 6]{bdd}) as follows.

\begin{lemma}\label{lemma 2.19}(Generalized Workhorse Lemma) Let $ R $ be a local ring such that $2 \in U(R) $, $n \geq 2$ and $A,E \in {\rm T}_n(R)$. Suppose that, for all $(i, j) \neq (1, n)$, $(E^3)_{ij} = E_{ij}$ and $(AE - EA)_{ij} = 0$. Suppose also that
\[
A = \begin{pmatrix}
a & \alpha & c \\
 & B & \beta \\
 & & b
\end{pmatrix}
\text{ and }
E = \begin{pmatrix}
e & \gamma & z \\
 & F & \delta \\
 & & f
\end{pmatrix}
\]
where $B,F \in {\rm T}_{n-2}(R)$, $a, b, c, e, f, z \in R$, $\alpha, \gamma \in {\rm M}_{1,n-2}(R)$ and $\beta, \delta \in {\rm M}_{n-2,1}(R)$. Then, the following items are fulfilled:

\noindent(i) Given $e = f = 1$, then $E^3 = E$ if, and only if, $z=-1/2(\gamma F\delta + 2\gamma\delta)$, and in this case, $AE = EA$.

\noindent(ii) Given $e = f = -1$, then $E^3 = E$ if, and only if, $z=-1/2(\gamma F\delta - 2\gamma\delta)$, and in this case, $AE = EA$.

\noindent(iii) Given $e = f = 0$, then $E^3 = E$ if, and only if, $z=\gamma F\delta$, and in this case, $AE = EA$.

\noindent(iv) Given $e = 1$ and $f = -1$, then $E^3 = E$. Further, $AE = EA$ if, and only if, $z$ satisfies the equation
$
az - zb = \gamma \beta - \alpha \delta + 2c.
$

\noindent(v) If $e = -1$ and $f = 1$, then $E^3 = E$. Further, $AE = EA$ if, and only if, $z$ satisfies the equation
$
az - zb = \gamma \beta - \alpha \delta - 2c.
$

\noindent(vi) If $e = 1$ and $f = 0$, then $E^3 = E$. Further, $AE = EA$ if, and only if, $z$ satisfies the equation
$
az - zb = \gamma \beta - \alpha \delta + c.
$

\noindent(vii) If $e = 0$ and $f = 1$, then $E^3 = E$. Further, $AE = EA$ if, and only if, $z$ satisfies the equation
$
az - zb = \gamma \beta - \alpha \delta - c.
$

\noindent(viii) If $e = -1$ and $f = 0$, then $E^3 = E$. Further, $AE = EA$ if, and only if, $z$ satisfies the equation
$
az - zb = \gamma \beta - \alpha \delta - c.
$

\noindent(w) If $e = 0$ and $f = -1$, then $E^3 = E$. Further, $AE = EA$ if, and only if, $z$ satisfies the equation
$
az - zb = \gamma \beta - \alpha \delta + c.
$
\end{lemma}

\begin{proof}
(i) It is apparent that $E^3 = E$ if, and only if, $z = -1/2(\gamma F\delta + 2\gamma\delta)$. We show that $AE = EA$. Given the assumptions, we have
\begin{align}
& z = -1/2(\gamma F\delta + 2\gamma\delta) \label{a1 1} \\
& \gamma F = -\gamma F^2 \label{a1 2}\\
& F\delta = -F^2\delta \label{a1 3}\\
& \alpha + \gamma B = a\gamma + \alpha F \label{a1 4}\\
& F \beta + \delta b = B \delta + \beta \label{a1 5}
\end{align}

In virtue of the above equations, we compute that
\begin{align*}
(EA)_{1n} &= c + \gamma \beta + zb \overset{(\ref{a1 1})}{=} c + \gamma \beta -1/2 \gamma F\delta b - \gamma \delta b, \\
&\overset{(\ref{a1 5})}{=} c + \gamma \beta + \gamma(F \beta - B \delta - \beta) + 1/2\gamma F(F \beta - B \delta - \beta), \\
&= c + \gamma \beta + \gamma F \beta - \gamma B \delta - \gamma \beta + 1/2\gamma F^2 \beta - 1/2 \gamma F B \delta - 1/2 \gamma F \beta, \\
&\overset{(\ref{a1 2})}{=} c - \gamma B \delta - 1/2 \gamma F B \delta.
\end{align*}

\begin{align*}
(AE)_{1n} &= az + \alpha \delta + c \overset{(\ref{a1 1})}{=} -a\gamma \delta - 1/2 a \gamma F \delta + \alpha \delta + c, \\
&\overset{(\ref{a1 4})}{=} (\alpha F - \alpha - \gamma B)\delta + 1/2 (\alpha F - \alpha - \gamma B)F \delta + \alpha \delta + c, \\
&= \alpha F \delta - \alpha \delta - \gamma B \delta + 1/2 \alpha F^2 \delta - 1/2 \alpha F \delta - 1/2 \gamma B F \delta + \alpha \delta + c, \\
&\overset{(\ref{a1 3})}{=} c - \gamma B \delta - 1/2 \gamma B F \delta = c - \gamma B \delta - 1/2 \gamma F B \delta.
\end{align*}

Note that, since $ (AE - EA)_{ij} = 0 $, we establish $ FB = BF $.

\medskip

(ii) The proof is similar to part (i).

(iii) It is obvious that $ E^3 = E $ if, and only if, $ z = \gamma F \delta $. We show that $ AE = EA $. Given the assumptions, we have
\begin{align}
&z = \gamma F \delta, \label{a2 1}\\
&\gamma = \gamma F^2,\label{a2 2}\\
&\delta = F^2 \delta,\label{a2 3}\\
&\gamma B = a \gamma + \alpha F,\label{a2 4}\\
&B \delta = F \beta + \delta b.\label{a2 5}
\end{align}

From the above equations, we calculate that

\begin{align*}
(EA)_{1n} &= \gamma \beta + zb \overset{(\ref{a2 1})}{=} \gamma \beta + \gamma F \delta b, \\
&\overset{(\ref{a2 5})}{=} \gamma \beta + \gamma F (B \delta - F \beta), \\
&= \gamma \beta + \gamma F B \delta - \gamma F^2 \beta, \\
&\overset{(\ref{a2 2})}{=} \gamma F B \delta.
\end{align*}
\begin{align*}
(AE)_{1n} &= az + \alpha \delta \overset{(\ref{a2 1})}{=} a \gamma F \delta + \alpha \delta, \\
&\overset{(\ref{a2 4})}{=} (\gamma \beta - \alpha F) F \delta + \alpha \delta,\\
&= \gamma B F \delta - \alpha F^2 \delta + \alpha \delta,\\
&\overset{(\ref{a2 3})}{=} \gamma B F \delta = \gamma F B \delta.
\end{align*}

(iv) Assume $ e = 1 $ and $ f = -1 $. Then, under the assumptions, we deduce that

$$\gamma F = -\gamma F^2, \quad F^2 \delta = F \delta \Longrightarrow \gamma F \delta = -\gamma F^2 \delta = -\gamma F \delta \Longrightarrow 2 \gamma F \delta = 0.$$

But, since $ 2 \in U(R) $, we have $ \gamma F \delta = 0 $. Thus, we get $ (E^3)_{1n} = \gamma F \delta + z = z = E_{1n} $ and, therefore, $ E^3 = E $. Moreover, it is clear that $ EA = AE $ if, and only if,

$$az + \alpha \delta - c = c + \gamma \beta + zb,$$

\noindent which is equivalent to

$$az - zb = \gamma \beta - \alpha \delta + 2c.$$

\medskip

(v) The proof is similar to part (iv).

(vi) Assume $ e = 1 $ and $ f = 0 $. So, under the given assumptions, we have

$$\gamma F = -\gamma F^2, \quad \delta = F^2 \delta \Longrightarrow \gamma F \delta = -\gamma F^2 \delta = -\gamma \delta.$$

Consequently, we derive $ (E^3)_{1n} = z + \gamma \delta + \gamma F \delta = z = E_{1n} $, and hence $ E^3 = E $. It is also readily checked that $ EA = AE $ if, and only if,

$$az - zb = \gamma \beta - \alpha \delta + c.$$

Finally, one sees that points (vii), (viii) and (w) possess proofs which are similar to that of (vi).
\end{proof}

The next preliminary facts are worthy of discussion: let $ a \in R $. The mappings $l_a: R \to R$ and $r_a: R \to R$ represent the (additive) abelian group endomorphisms defined respectively by $l_a(r) = ar$ and $r_a(r) = ra$ for all $r \in R$. Consequently, the expression $ l_a - r_b $ defines an abelian group endomorphism such that $(l_a - r_b)(r) = ar - rb$ for any $r \in R$. According to \cite{diesl}, a local ring $ R $ is classified as {\it bleached} if, for any $ a \in U(R) $ and $ b \in J(R) $, both $ l_a - r_b $ and $ l_b - r_a $ are surjective. The category of bleached local rings includes many well-established examples, such as commutative local rings, local rings with nil Jacobson radicals, and local rings in which some power of each element of their Jacobson radicals is central (see \cite[Example 13]{bdd}).

\medskip

Now, we need the following.

\begin{lemma} \label{there ex tri}
Let $R$ be a local ring such that $2 \in U(R)$, and suppose that $A \in {\rm T}_n(R)$. Write $A$ as $(a_{ij})$. Then, for any set $\{e_{ii}\}_{i=1}^n$ of tripotents in $R$ such that $e_{ii} = e_{jj}$ whenever $l_{a_{ii}} - r_{a_{jj}}$ is not a surjective abelian group endomorphism of $R$, there exists a tripotent $E \in {\rm T}_n(R)$ such that $AE = EA$ and $E_{ii} = e_{ii}$ for every $i \in \{1,\dots,n\}$.
\end{lemma}

\begin{proof}
Leveraging Lemma \ref{lemma 2.19}, the proof process mirrors that of \cite[Lemma 7]{bdd}. To avoid redundancy, we omit the detailed proof.
\end{proof}

We are now in a position to attack the main result in this section, in which proof we shall apply the established above Theorem~\ref{boolean and yaqub}.

\begin{theorem}\label{local}
Let $R$ be a local ring and $n > 2$. Then, the following conditions are equivalent:

(1) ${\rm T}_n(R)$ is an SDT ring.

(2) either

(2.1) \( R \) is a bleached ring and \( R/J(R) \cong \mathbb{Z}_2 \),

\noindent or

(2.2) \( R \) is a bleached ring, \( R/J(R) \cong \mathbb{Z}_3 \) and, if \( a, b \in R \) such that \( a - 1 \in \Delta(R) \) and \( b + 1 \in \Delta(R) \), then \( l_a - r_b : R \to R \) is surjective.
\end{theorem}

\begin{proof}
Since $R$ is a local ring, we have either $2 \in J(R)$ or $2 \in U(R)$. We prove the theorem for both cases independently.

\medskip

\noindent \textbf{Case 1:} If $2 \in J(R)$.

\medskip

\noindent $(1) \Rightarrow (2.1)$. Since $2 \in J(R)$, Theorem~\ref{boolean and yaqub} illustrates that $R/J(R)$ is a Boolean ring. But, since $R$ is local, we must have $R/J(R) \cong \mathbb{Z}_2$. Because ${\rm T}_n(R)$ is an SDT ring, Corollary \ref{corner ring} gives that ${\rm T}_2(R)$ is too an SDT ring. Moreover, Proposition \ref{pro 4} allows us to detect that $T_2(R)$ is an SDI ring.

Suppose now $a \in U(R)$ and $b \in J(R)$. We intend to show that $l_a - r_b : R \to R$ is surjective. Thereby, it suffices to prove that, for every $v \in R$, there exists $x \in R$ such that $ax - xb = v$. Put $r := \begin{pmatrix} a & v \\ 0 & b \end{pmatrix}$. Assume $r = g + j$ is an SDI representation, where $g = \begin{pmatrix} e & x \\ 0 & f \end{pmatrix}$ and $j = \begin{pmatrix} d & y \\ 0 & d' \end{pmatrix}$. Since $e$ is an idempotent and $a \in U(R)$, we deduce $e = 1$. However, since $f$ is an idempotent and $b \in J(R)$, we derive $f = 0$. Thus, $g = \begin{pmatrix} 1 & x \\ 0 & 0 \end{pmatrix}$. Since $rg = gr$, we now have $ax - xb = v$. Therefore, $l_a - r_b : R \to R$ is surjective. Similarly, we can show that $l_b - r_a : R \to R$ is surjective, as desired.

\noindent $(2.1) \Rightarrow (1)$. Since $2 \in J(R)$, we only have the case $R/J(R) \cong \mathbb{Z}_2$. Thus, by \cite[Theorem 4.4]{csj}, there is nothing left to prove.

\medskip

\noindent \textbf{Case 2:} If $2 \in U(R)$.

\medskip

\noindent $(1) \Rightarrow (2.2)$. Since $2 \in U(R)$, Theorem \ref{boolean and yaqub} demonstrates that $R/J(R)$ is a Yaqub ring. But, since $R$ is local, we must have $R/J(R) \cong \mathbb{Z}_3$. Because ${\rm T}_n(R)$ is an SDT ring, Corollary \ref{corner ring} gives that ${\rm T}_2(R)$ is too an SDT ring.

Suppose now $a \in U(R)$ and $b \in J(R)$. We intend to show that $l_a - r_b : R \to R$ is surjective. Thereby, it suffices to establish that, for each $v \in R$, there is $x \in R$ such that $ax - xb = v$. Set $r := \begin{pmatrix} a & v \\ 0 & b \end{pmatrix}$. Assume $r = g + j$ is an SDT representation, where $g = \begin{pmatrix} e & x \\ 0 & f \end{pmatrix}$ and $j = \begin{pmatrix} d & y \\ 0 & d' \end{pmatrix}$. Since $b \in J(R)$ and $f$ is a tripotent, we detect $f = 0$. On the other hand, Lemma \ref{trivial tripotent} allows us to conclude that $R$ has no non-trivial tripotents. Hence, since $a \in U(R)$, $e$ is simultaneously a unit and a tripotent element, and thus either $e = 1$ or $e = -1$. If $g = \begin{pmatrix} 1 & x \\ 0 & 0 \end{pmatrix}$, then since $rg = gr$, we have $ax - xb = v$. If, however, $g = \begin{pmatrix} -1 & x \\ 0 & 0 \end{pmatrix}$, then again since $rg = gr$, we have $a(-x) - (-x)b = v$. Consequently, $l_a - r_b : R \to R$ is surjective. Similarly, we can establish that $l_b - r_a : R \to R$ is surjective.

We now show that under the given assumptions, the SDT representation of elements is unique. In this light, suppose \( e + d = f + b \) are two SDT representations in \( R \). Note that, Lemma \ref{trivial tripotent} manifestly yields \( e, f \in \{-1, 0, 1\} \), so that one easily sees that either \( e = f \) or \( e = -f \). If \( e = -f \), then \( 2e = b - d \in \Delta(R) \). Since \( 2 \in U(R) \), we have \( e \in \Delta(R) \). Thus, \( e^2 \in \Delta(R) \cap \text{Id}(R) = \{0\} \), which leads to \( e = 0 \). Therefore, \( e = f = 0 \).

Suppose now that \( a = 1 + d \) and \( b = -1 + d' \) are two SDT representations. Assume that

\[
r = \begin{pmatrix}
a & v \\
0 & b
\end{pmatrix}
\]

\noindent is an element of \( {\rm T}_2(R) \). Also, suppose that \( r = g + w \) is an SDT representation, where

\[
g = \begin{pmatrix}
e & x \\
0 & f
\end{pmatrix}
\quad \text{and} \quad
w = \begin{pmatrix}
d & y \\
0 & d'
\end{pmatrix}.
\]

Bearing in mind the above note, we can assume without loss of generality that \( e = 1 \) and \( f = -1 \). Since \( gw = wg \) and \( 2 \in U(R) \), we deduce \( a(1/2)x - (1/2)xb = v \). This obviously implies that the map \( l_a - r_b : R \to R \) is surjective.

\medskip

\noindent $(2.1) \Rightarrow (1)$. Suppose \( A \in {\rm T}_n(R) \). We show that \( A \) has an SDT representation such that \( A = E + D \) in \( {\rm T}_n(R) \). Since \( R/J(R) \cong \mathbb{Z}_3 \), we see with no any technical difficulty that \( R = J(R) \cup (1 + J(R)) \cup (-1 + J(R)) \). First, we construct the elements on the main diagonal \( E \). Suppose

\[e_{ii} :=
\begin{cases}
\: \: \: 0 \quad \text{if } a_{ii} \in J(R),\\
\: \: \: 1 \quad \text{if } a_{ii} \in 1+J(R),\\
-1 \quad \text{if } a_{ii} \in -1+J(R).
\end{cases}
\]

Therefore, one inspects that \( a_{ii} - e_{ii} \in J(R) \) for each \( i \). Notice that, since \( 2 \in U(R) \), it must be that \( (1+J(R)) \cap (-1 + J(R)) = \emptyset \). If \( e_{ii} \neq e_{jj} \), then we come to

\[
\begin{cases}
(1)\: e_{ii} \in U(R) \: \text{and} \: e_{jj} \in J(R),\\
(2)\: e_{ii} \in J(R) \: \text{and} \: e_{jj} \in U(R),\\
(3)\: e_{ii} \: \text{and} \: e_{jj} \in U(R).
\end{cases}
\]

We prove that, in all three cases, \( l_{a_{ii}} - r_{a_{jj}}: R \to R \) is necessarily surjective.

In fact, for case (1), \( a_{ii} \in U(R) \) and \( a_{jj} \in J(R) \) and, because \( R \) is bleached, \( l_{a_{ii}} - r_{a_{jj}}: R \to R \) is indeed surjective.

The case (2) is observed to be similar to case (1).

In case (3), without harm of generality, assuming \( e_{ii} = 1 \) and \( e_{jj} = -1 \), we obtain that \( a_{ii} - 1, a_{jj} + 1 \in \Delta(R) \). Therefore, by the requested assumption, \( l_{a_{ii}} - r_{a_{jj}}: R \to R \) is surjective. Hence, with Lemma \ref{there ex tri} in hand, there is an tripotent \( E \in T_n(R)\) such that \( AE = EA \) and \( E_{ii} = e_{ii} \) for each \( i \in \{1, \ldots, n\} \). In addition, $$A - E \in J(T_n(R))\subseteq \Delta(T_n(R)),$$ thus completing the proof.
\end{proof}

The case when $n=2$ can be considered separately in the following manner.

\begin{example} Suppose $R$ is an integral domain and an SDT ring. Then, ${\rm T}_2(R)$ is an SDT ring.
\end{example}

\begin{proof} Utilizing Proposition~\ref{local}, $R$ is a local ring. In the other vein, since $R$ is a domain, arguing as in the proof of Proposition~\ref{local}, we can assume that $R$ has no non-trivial tripotents.

Since $R$ is local, we have either $2 \in U(R)$ or $2 \in J(R)$. First, we assume that $2 \in J(R)$, and let $A = \begin{pmatrix} a & \beta \\ 0 & b \end{pmatrix} \in {\rm T}_2(R)$. Note that an SDT ring with $2 \in J(R)$ is always an SDI ring. We show that ${\rm T}_2(R)$ is also SDI. Precisely, we consider the following four cases:

\begin{enumerate}
    \item If $a, b \in J(R)$, then $A \in J(R)$, so $A = 0 + A$ is an SDI representation.
    \item If $a, b \in U(R)$, then since $R$ is both SDI and local, we have $a - 1 \in J(R)$ and $b - 1 \in J(R)$. Therefore, $$A = I_2 + \begin{pmatrix} a-1 & \beta \\ 0 & b-1 \end{pmatrix}$$
        \noindent is an SDI representation for $A$.
\end{enumerate}

\begin{enumerate}
    \setcounter{enumi}{2}
    \item $a \in U(R), b \in J(R)$. Since $R$ is an SDI ring, we obtain $a-1 \in J(R)$. Thus,
    $$A = \begin{pmatrix} 1 & \alpha \\ 0 & 0 \end{pmatrix} + \begin{pmatrix} a-1 & \beta - \alpha \\ 0 & b \end{pmatrix}$$
    \noindent is an SDT representation, where $\alpha = \beta((a-1)+(1-b))^{-1}$.

    \item $b \in U(R), a \in J(R)$. Since $R$ is an SDI ring, we receive $b-1 \in J(R)$. So,
    $$A = \begin{pmatrix} 0 & \alpha \\ 0 & 1 \end{pmatrix} + \begin{pmatrix} a & \beta - \alpha \\ 0 & b-1 \end{pmatrix}$$
    \noindent is an SDT representation, where $\alpha = \beta((b-1)+(1-a))^{-1}$.
\end{enumerate}

Now, suppose $2 \in U(R)$.
\begin{enumerate}
    \setcounter{enumi}{0}
    \item If $a,b \in J(R)$, then $A \in J(R)$, so $A = 0+A$ is an SDT representation.

    \item Given $a, b \in U(R)$.
    If the SDT representations of $a$ and $b$ are of the form $a=1+(a-1)$ and $b=1+(b-1)$, then $$A=I_2+ \begin{pmatrix} a-1 & \beta\\ 0 & b-1 \end{pmatrix}$$
    \noindent is an SDT representation for $A$.

    If the SDT representations of $a$ and $b$ are of the form $a=-1+(a+1)$ and $b=-1+(b+1)$, then $$A=-I_2+ \begin{pmatrix} a+1 & \beta\\ 0 & b+1 \end{pmatrix}$$
    \noindent is an SDT representation for $A$.

    If the SDT representations of $a$ and $b$ are of the form $a=-1+(a+1)$ and $b=1+(b-1)$, then
    $$A= \begin{pmatrix} -1 & \alpha \\ 0 & 1 \end{pmatrix} + \begin{pmatrix} a+1 & \beta - \alpha\\ 0 & b-1 \end{pmatrix}$$
     \noindent is an SDT representation, where $\alpha = 2 \beta (2+(b-1)-(a+1))^{-1}$.

    If the SDT representations of $a$ and $b$ are of the form $a=1+(a-1)$ and $b=-1+(b+1)$, then
    $$A= \begin{pmatrix} 1 & \alpha \\ 0 & -1 \end{pmatrix} + \begin{pmatrix} a-1 & \beta - \alpha\\ 0 & b+1 \end{pmatrix}$$
    \noindent is an SDT representation, where $\alpha = 2 \beta (2+(a-1)-(b+1))^{-1}$. Note that $2 \in U(R)$ is assumed.

    \item Given $a \in U(R)$ and $b \in J(R)$. If the SDT representation of $a$ is of the form $a = 1 + (a-1)$, then $$A = \begin{pmatrix} 1 & \alpha \\ 0 & 0 \end{pmatrix} + \begin{pmatrix} a-1 & \beta - \alpha \\ 0 & b \end{pmatrix}$$
    \noindent is an SDT representation for $A$, where $\alpha = \beta ((1-b) - (1-a))^{-1}$.

    If the SDT representation of $a$ is of the form $a = -1 + (a+1)$, then
    $$A = \begin{pmatrix} -1 & \alpha \\ 0 & 0 \end{pmatrix} + \begin{pmatrix} a+1 & \beta - \alpha \\ 0 & b \end{pmatrix}$$
    \medskip is an SDT representation for $A$, where $\alpha = \beta ((1+b) - (1+a))^{-1}$.

    \item Given $a \in J(R)$ and $b \in U(R)$. If the SDT representation of $b$ is of the form $b = 1 + (b-1)$, then $$A = \begin{pmatrix} 0 & \alpha \\ 0 & 1 \end{pmatrix} + \begin{pmatrix} a & \beta - \alpha \\ 0 & b-1 \end{pmatrix}$$
    \noindent is an SDT representation for $A$, where $\alpha = \beta ((b-1) + (1-a))^{-1}$.

    If the SDT representation of $b$ is of the form $b = -1 + (b+1)$, then
    $$A = \begin{pmatrix} 0 & \alpha \\ 0 & -1 \end{pmatrix} + \begin{pmatrix} a & \beta - \alpha \\ 0 & b+1 \end{pmatrix}$$
    \noindent is an SDT representation for $A$, where $\alpha = \beta ((1+a) - (1+b))^{-1}$, as claimed.
\end{enumerate}
\end{proof}

Now, we manage to examine the above stated example in a more general situation like the following one.

\begin{proposition} Let \( R \) be a ring that has no non-trivial tripotent elements. Then, the following conditions are equivalent:

(1) \( T(R,V) \) is an SDT ring.

(2) Either \( R/J(R) \cong \mathbb{Z}_2 \) or \( R/J(R) \cong \mathbb{Z}_3 \).
\end{proposition}

\begin{proof} "\( (1) \Rightarrow (2) \)": If \( T(R,V) \) is an SDT ring, it is easily verified that \( R \) is also an SDT ring. Moreover, since \( R \) has no non-trivial tripotent elements, as shown in Proposition~\ref{local}, we can prove that \( R \) is a local ring. Therefore, according to a combination of the locality of \( R \) and Theorem~\ref{boolean and yaqub}, we conclude \( R/J(R) \cong \mathbb{Z}_2 \) or \( R/J(R) \cong \mathbb{Z}_3 \).

\medskip

"\( (2) \Rightarrow (1) \)": If \( R/J(R) \cong \mathbb{Z}_2 \), then from \cite[Theorem 15]{nzu} we deduce that \( T(R,V) \) is a uniquely clean ring. Thus, it is an SDI ring and, consequently, an SDT ring.

If, however, \( R/J(R) \cong \mathbb{Z}_3 \), we so derive \[ R = J(R) \cup (1 + J(R)) \cup (-1 + J(R)) .\]

Assume now that \( \begin{pmatrix} a & v \\ 0 & a \end{pmatrix} \in T(R,V) \) is fulfilled. So, we have:

\medskip

(a) If \( a \in J(R) \), then \( \begin{pmatrix} a & v \\ 0 & a \end{pmatrix} \in J(T(R,V)) \).

\medskip

(b) If \( a \in 1 + J(R) \), then \[ \begin{pmatrix} a & v \\ 0 & a \end{pmatrix} = I_2 + \begin{pmatrix} a - 1 & v \\ 0 & a - 1 \end{pmatrix} ,\]

\noindent which is an SDT representation.

\medskip

(c) If \( a \in -1 + J(R) \), then \[ \begin{pmatrix} a & v \\ 0 & a \end{pmatrix} = -I_2 + \begin{pmatrix} a + 1 & v \\ 0 & a + 1 \end{pmatrix} ,\]

\noindent which is an SDT representation, as claimed.
\end{proof}

We finish our examinations with the following exhibitions which we leave to the interested reader for a direct check.

\begin{example}
Let \( R \) be a ring in which all tripotent elements are central. Then, the following issues hold:

(1) \( R \) is an SDT ring if, and only if, \( R[[x]] \) is an SDT ring.

(2) \( R \) is an SDT ring if, and only if, \( R[x]/(x^n) \) is an SDT ring.

(3) \( R \) is an SDT ring if, and only if, \( T(R,R) \) is an SDT ring.
\end{example}

\section{Concluding Discussion and Questions}

As above noticed, in \cite{kyzr} the authors defined and investigated those rings $R$, calling them {\it semi-tripotent}, whose elements are a sum of a tripotent element from $R$ and an element from the Jacobson radical of $R$ which, generally, need {\it not} commute each other.

\medskip

Now, regarding Proposition~\ref{semitrip}, one may ask are the classes of semi-tripotent rings and SDT rings independent each other; that is, does there exist an SDT ring which is {\it not} semi-tripotent as well as a semi-tripotent ring that is {\it not} SDT? However, it was proved in \cite[Theorem 3.5 (6)]{kyzr} that $R/J(R)$ has the same presentation as in our Theorem~\ref{boolean and yaqub} plus the requirement that all idempotents of $R$ lift modulo $J(R)$. That is why, it quite surprisingly follows that {\it every SDI ring whose idempotent lift modulo the Jacobson radical is always semi-tripotent}. However, as the opposite claim of Theorem~\ref{boolean and yaqub} is not at all guaranteed in order to be a satisfactory criterion, we do {\it not} know yet if any semi-tripotent ring is SDT. Likewise, due to the lifting restriction of the idempotents, the reciprocal implication {\it cannot} be happen in all generality or, in other words, there is an SDT ring that is {\it not} semi-tripotent.

\medskip

Furthermore, in regard to Corollary~\ref{corner ring}, a logically arising question is whether or not the converse in its formulation holds. So, one may formulate our first challenging question.

\medskip

\noindent{\bf Problem 1.} Suppose $e\in R$ is an idempotent. If both $eRe$ and $(1-e)R(1-e)$ are SDT rings, is it true that so does $R$?

\medskip

Our second intriguing query is related to the study in-depth of a generalized version of the SDT rings like this, which presents a more general setting of the {\it semi-n-potent rings} as defined in \cite{kyzr}.

\medskip

\noindent{\bf Problem 2.} Describe those rings $R$, naming them {\it strongly $\Delta$ n-potent}, whose elements are a sum of a $n$-potent element in $R$ (i.e., an element $a\in R$ such that $a^n=a$ for some $n\in \mathbb{N}$) and an element from $\Delta(R)$ that commute each other.

\medskip

On the other side, in conjunction with \cite{MWL}, we close our work with the following interesting question.

\medskip

\noindent{\bf Problem 3.} Characterize those rings $R$, calling them {\it C$\Delta$ rings}, whose elements are a sum of an element from the center $\mathrm{Z}(R)$ and from $\Delta(R)$.

\medskip

\noindent{\bf Funding:} The first and second authors are supported by Bonyad-Meli-Nokhbegan and receive funds from this foundation.

\end{document}